\documentclass[aap,preprint,oneside,12pt]{article}

\usepackage{mathrsfs,eucal}
\usepackage{pstricks}
\usepackage{amsthm,amssymb,amsmath,amscd}
\usepackage{graphics}
\usepackage{graphicx}
\usepackage{epstopdf}
\usepackage{color}
\usepackage{url}
\usepackage{psfrag}
\usepackage{pifont}
\usepackage{rotating,booktabs}
\usepackage{enumerate}
\usepackage[all,cmtip]{xy}
\usepackage[bf]{caption}
\usepackage{fancyhdr,fancybox}

\numberwithin{equation}{section}

\usepackage{fourier}

\usepackage{natbib}
\bibpunct{(}{)}{;}{a}{,}{,}
\usepackage{hyperref}
\hypersetup{
   colorlinks=true,       
   linkcolor=medblue,          
   citecolor=medblue,        
   filecolor=medblue,      
   urlcolor=medblue           
}

\usepackage[margin=3cm]{geometry}

\definecolor{newgreen}{rgb}{0.1, 0.6, 0.1}
\definecolor{newblue}{rgb}{0.0, 0.1, 0.7}
\definecolor{ggreen}{rgb}{0.5, 0.85, 0.3}
\definecolor{rred}{rgb}{0.65, 0.2, 0.2}
\definecolor{ggray}{gray}{0.7}
\definecolor{bblue}{rgb}{0.0, 0.0, 1}
\definecolor{darkbrown}{rgb}{0.4, 0.26, 0.13}
\definecolor{medblue}{rgb}{0,0,.9}

\newtheorem{defi}{Definition}[section]
\newtheorem{rem}[defi]{Remark}
\newtheorem{lem}[defi]{Lemma}
\newtheorem{cor}[defi]{Corollary}
\newtheorem{pro}[defi]{Proposition}
\newtheorem{theo}[defi]{Theorem}
\newtheorem{exm}[defi]{Example}

\newcommand{\prob}{\mathbf{P}}
\newcommand{\esp}{\mathbf{E}}

\newcommand{\R}{\mathbb R}

\newcommand{\e}{\varepsilon}

\newcommand{\et}{\quad\mbox{and}\quad}
\newcommand{\ou}{\quad\mbox{where}\quad}

\newcommand{\bc}{{\bf c}}
\newcommand{\bcs}{{\bf c^\star}}
\newcommand{\supp}{\textrm{supp}}
\newcommand{\Qk}{\mathfrak Q_k}

\begin{document}

%
%
%
%
%

\title{A notion of stability for k-means clustering}
\author{Thibaut Le Gouic \footnote{Aix Marseille Univ, CNRS, Centrale Marseille, I2M, Marseille, France} and Quentin Paris\footnote{National Research University Higher School of Economics. The study has been funded by the Russian Academic Excellence Project 5-100}}

\maketitle

\begin{abstract}
In this paper, we define and study a new notion of stability for the $k$-means clustering scheme building upon the field of quantization of a probability measure. We connect this definition of stability to a geometric feature of the underlying distribution of the data, named absolute margin condition, inspired by recent works on the subject. 
\end{abstract}

\tableofcontents

\newpage
\section{Introduction} 
Unsupervised classification consists in partitioning a data set into a series of groups (or clusters) each of which may then be regarded as a separate class of observations.
This task, widely considered in data analysis, enables, for instance, practitioners, in many disciplines, to get a first intuition about their data by identifying meaningful groups of observations.
The tools available for unsupervised classification are various.
Depending on the nature of the problem, one may rely on a model based strategy modeling the unknown distribution of the data as a mixture of known distributions with unknown parameters.
Another approach, model-free, is embodied by the well known $k$-means clustering scheme. This paper focuses on the stability of this clustering scheme.

\subsection{Quantization and the k-means clustering scheme}
\label{kmeans}
The $k$-means clustering scheme prescribes to classify observations according to their distances to chosen representatives.
This clustering scheme is strongly connected to the field of quantization of probability measures and this paragraph shortly recalls how these concepts interact.
Suppose our data modeled by $n$ i.i.d. random variables $X_{1},\dots,X_{n}$, taking their values in some metric space $(E,d)$, and with same distribution $P$ as (and independent of) a generic random variable $X$.
Let $k\ge 1$ be an integer fixed in advance, representing the prescribed number of clusters, and define a $k$-points\,\footnote{The integer $k$ is supposed fixed throughout the paper and all quantizers considered below are supposed to be $k$-points quantizers.} quantizer as any mapping $q:E\to E$ such that\footnote{For a set $A$, notation $|A|$ refers to the number of elements in $A$.} $|q(E)|=k$. 
Denoting $c_1,\dots,c_k$ the values taken by $q$, the sets $\{x\in E: q(x)=c_j\}$, $1\le j\le k$, partition the space $E$ into $k$ subsets (or cell) and each point $c_j$ (called indifferently a center, a centroid or a code point) stands as a representative of all points in its cell.
Given a quantizer $q$, associated data clusters are defined, for all $1\le j\le k$, by
\[
C_{j}(q):=\{x\in E: q(x)=c_j\}\cap\{X_{1},\dots,X_{n}\}.
\] 

The performance of this clustering scheme is naturally measured by the average square distance, with respect to $P$, of a point to its representative. In other words, the risk of $q$ (also referred to as its distortion) is defined by
\begin{equation}
\label{distortionq}
R(q):=\int_E d(x,q(x))^2\,{\rm d}P(x).
\end{equation}
Quantizers of special interest are nearest neighbor (NN) quantizers, \emph{i.e.} quantizers such that, for all $x\in E$,
\[q(x)\in\underset{c\in q(E)}{\arg\min}\ d(x,c).\]
The interest for these quantizers relies on the straightforward observation that for any quantizer $q$, an NN quantizer $q'$ such that $q(E)=q'(E)$ satisfies $R(q')\le R(q)$. 
Hence, attention may be restricted to NN quantizers and any optimal quantizer 
\begin{equation}
\label{qopt}
q^{\star}\in \underset{q}{\arg\min}\ R(q),
\end{equation}
(where $q$ ranges over all quantizers $k$-points quantizers) is necessarily an NN quantizer.
We will denote $\Qk$ the set of all $k$-points NN quantizers and, unless mentionned explicitly, all quantizers involved in the sequel will be considered as members of $\Qk$. For $q\in\Qk$, the value of its risk is entirely described by its image. Indeed, if $q\in\Qk$ takes values $c_1,\dots,c_k$, then 
\begin{equation}\label{Rc}
R(q)=\int_E\min_{1\le j\le k}d(x,c_{j})^2\,{\rm d}P(x).
\end{equation}
Denoting $\mathbf c=\{c_1,\dots,c_k\}$, referred to as a codebook, we will often denote by $R(\mathbf c)$ the right hand side of \eqref{Rc} with a slight abuse of notation.\\

A few additional considerations, relative to NN-quantizers, will be useful in the paper. Given ${\bf c}=\{c_1,\dots,c_k\}$, denote $V_j({\bf c})$ the set of points in $E$ closer to $c_j$ than to any other $c_\ell$, that is
\[V_j({\bf c}):=\left\{\,x\in E\,:\,\forall \ell\in\{1,\dots,k\}\,,\, d(x,c_{j})\le d(x, c_{\ell})\,\right\}.\]
These sets do not partition the space $E$ since, for $i\ne j$, the set $V_i({\bf c})\cap V_j({\bf c})$ is not necessarily empty. A Voronoi partition of $E$ relative to ${\bf c}$ is any partition $W_1,\dots, W_k$ of $E$ such that, for all $1\le j\le k$, $W_j\subset V_j(\mathbf c)$ up to relabeling. For instance, given $q\in\Qk$ with image $\mathbf c$, the sets $W_j=q^{-1}(c_j)$, $1\le j\le k$, form a Voronoi partition relative to $\mathbf c$. We call frontier of the Voronoi diagram generated by $\bf c$ the set
\begin{equation}
\label{front}
\mathcal F({\bf c}):=\bigcup_{i\ne j}V_i({\bf c})\cap V_j({\bf c}).
\end{equation}
Given an optimal quantizer $q^{\star}$ with image $\mathbf c^{\star}=\{c^{\star}_1,\dots,c^{\star}_k\}$, a remarkable property, known as the center condition, states that for all $1\le j\le k$, and provided $|\supp(P)|\ge k$,
\begin{equation}
\label{ccond}
P(V_j(\mathbf c^{\star}))>0\et c^{\star}_j\in\underset{c\in E}{\arg\min}\int_{V_j(\mathbf c^{\star})}d(x,c)^2\,\textrm{d}P(x).
\end{equation}
From now on, the probability measure $P$ will be supposed to have a support of more than $k$ points.\\


We end this subsection by mentioning that computing an optimal quantizer requires the knowledge of the distribution $P$. From a statistical point of view, when the only information available about $P$ consists in the sample $X_1,\dots,X_n$, reasonable quantizers are empirically optimal quantizers, \emph{i.e.} NN quantizers associated to any codebook $\hat{\mathbf c}=\{\hat c_{1},\dots,\hat c_{k}\}$ satisfying
\begin{equation}
\label{empq}
\hat{\mathbf c}\in\underset{\mathbf c=\{c_1,\dots,c_k\}}{\arg\min}\ \hat R(\mathbf c)\ou \hat R(\mathbf c)=\frac{1}{n}\sum_{i=1}^{n}\min_{1\le j\le k}d(X_i,c_j)^2.
\end{equation}
In other words, empirically optimal quantizers minimize the risk associated to the empirical measure 
\[P_n:=\frac{1}{n}\sum_{i=1}^{n}\delta_{X_i}.\]
\noindent The computation of empirically optimal centers is known to be a hard problem, due in particular to the non-convexity of ${\bf c}\mapsto \widehat R(\mathbf c)$, and is usually performed by Lloyd's algorithm for which convergence guarantees have been obtained recently by \cite{LuZh16} in the context where $P$ is a mixture of sub-gaussian distributions. 

\subsection{Risk bounds}
\label{risk}
The performance of the $k$-means clustering scheme, based on the notion of risk, has been widely studied in the literature. Whenever $(E,|.|)$ is a separable Hilbert space, the existence of an optimal codebook, \emph{i.e.} of $\mathbf c^{\star}=\{c^{\star}_1,\dots,c^{\star}_k\}$ such that  
\[
R(\mathbf c^{\star})=R^{\star}=\inf_{\mathbf c=\{c_1,\dots,c_k\}}R(\mathbf c),
\]
is well established \citep[see, e.g, Theorem 4.12 in][]{GrLu00}, provided $\esp|X|^2<+\infty$.
In this same context, works of \citet{Po81,Po82a} and \citet{AbWi84} imply that $R(\hat{\mathbf c})\to R^{\star}$ almost surely as $n$ goes to $+\infty$, where $\hat{\mathbf c}$ is as in \eqref{empq}.
The non-asymptotic performance of the $k$-means clustering scheme has also received a lot of attention and has been studied, for example, by \citet{Ch94, LiLuZe94, BaLiLu98, Li00, Li01, An05, AnGyGy05} and \citet{BiDeLu08}.
For instance \citet{BiDeLu08} prove that in a separable Hilbert space, and provided $|X|\le L$  almost surely, then
\[
\esp R(\hat{\mathbf c})-R^{\star}\le 12kL^{2}/\sqrt n,
\]
for all $n\ge 1$. A similar result is established in \cite{CaPa12} relaxing the hypothesis of bounded support by supposing only the existence of an exponential moment for $X$.
In the context of a separable Hilbert space, \citet{Le15} establishes a stronger result under some conditions involving the quantity $p(t)$ defined as follows. 

\begin{defi}[\,{\citealp{Le15}}\,]\label{def:marglev} Let $\mathcal M$ be the set of all ${\bf c^{\star}}=\{c^{\star}_1,\dots,c^{\star}_k\}$ such that $R(\mathbf c^{\star})=R^{\star}$.
For $t\ge 0$, we define
\begin{equation}
\label{pe}
p(t):=\sup_{{\bf c^{\star}}\in\mathcal M} P(\mathcal F({\bf c^{\star}})^{t}),
\end{equation}
where, for any set $A\subset E$, the notation $A^{t}$ stands for the $t$-neighborhood of $A$ in $E$ defined by $A^{t}=\{x\in E:d(x,a)\le t\}$ and where $\mathcal F({\bf c^{\star}})$ is defined in \eqref{front}.
\end{defi}
\noindent For any codebook ${\bf c}=(c_1,\dots,c_k)$, $P(\mathcal F({\bf c})^{t})$ corresponds to the probability mass of the frontier  of the associated Voronoi diagram inflated by $t$ (see Figure \ref{margin}). Under some slight restrictions and supposing $p(t)$ does not increase too rapidly with $t$, it appears that the excess risk is of order $\mathcal O(1/n)$ as described below.
\begin{theo}[\,{Proposition 2.1 and Theorem 3.1 in \citealp{Le15}}\,]
\label{theoLevrard}
Suppose that $(E,|.|)$ is a (separable) Hilbert space. Denote 
\[B=\inf_{\mathbf c^{\star}\in\mathcal M,i\ne j}|c^{\star}_i-c^{\star}_j|\et p_{\emph{min}}=\inf_{\mathbf c^{\star}\in\mathcal M, 1\le j\le k}P(V_j({\mathbf c}^{\star})).\]
$(1)$ Suppose that $P(x: |x|\le L)=1$ for some $L>0$. Then $B>0$ and $p_{\emph {min}}>0$.\vspace{0.2cm}\\ 
$(2)$ Suppose in addition that there exists $r_0>0$ such that, for all $0<t\le r_0$,
\[p(t)\le \frac{Bp_{\emph{min}}}{128 L^2}t,\]
where $p(t)$ is as in \eqref{pe}. Then, for all $x>0$, and any $\mathbf{\hat c}$ minimizing the empirical risk as in \eqref{empq}, 
\[R(\mathbf{\hat c})-R^{\star}\le \frac{C(k+x)L^2}{n},\]
with probability at least $1-e^{-x}$, where $C>0$ denotes a constant depending on auxiliary (and explicit) characteristics of $P$.
\end{theo}

\begin{figure}[htbp]
\centering
\includegraphics[width=8cm,angle=-90]{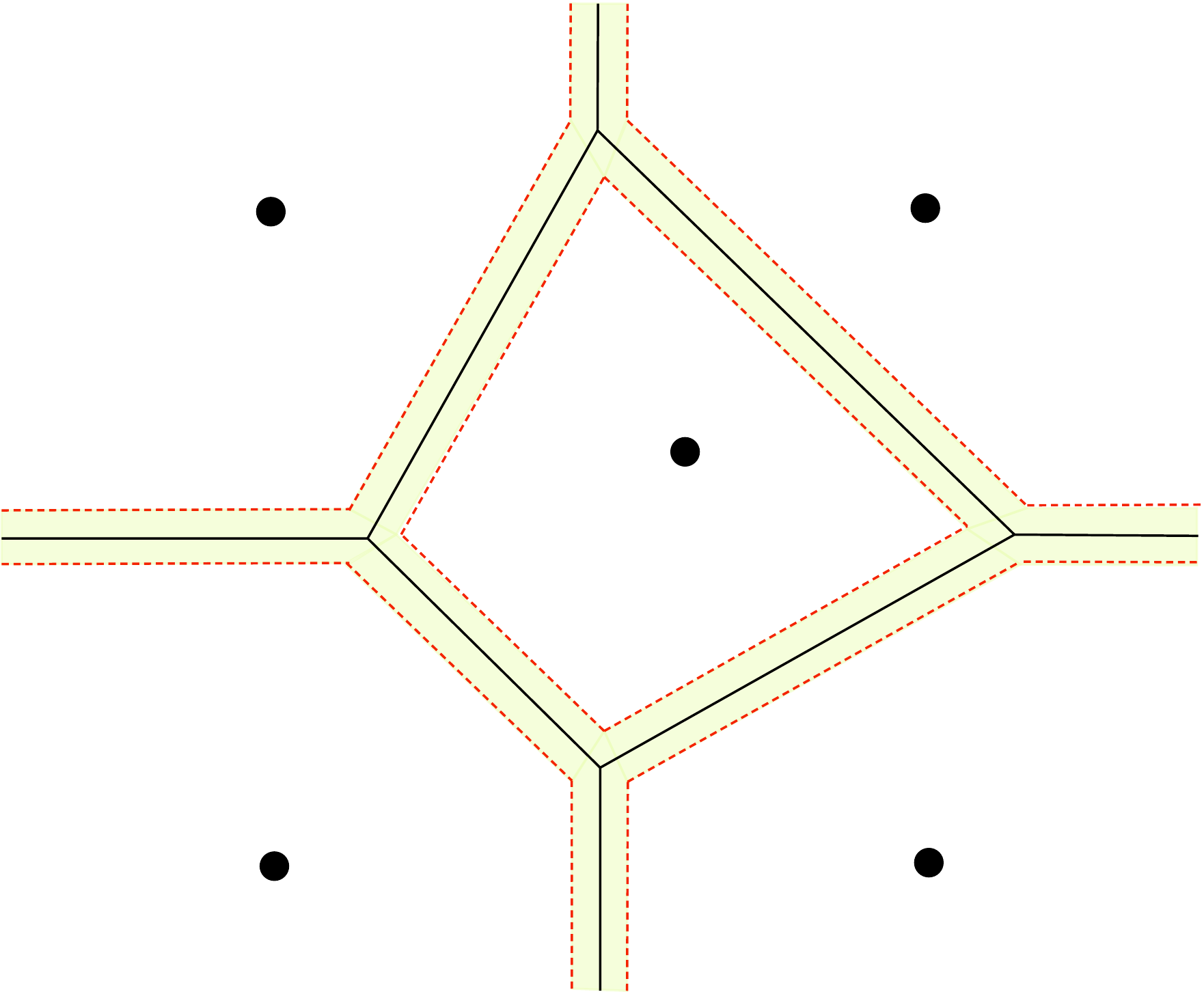}
\caption{For $k=5$, the figure represents $k$ centers in the Euclidean plane. The black solid lines define the frontier of the associated Voronoi diagram. The light-green area, inside the red dashed lines, corresponds to the $t$-neighborhood of this frontier for some small $t$.}
\label{margin}
\end{figure}

\subsection{Stability}\label{sec:introstab}
For a quantizer $q\in\Qk$, the risk $R(q)$ describes the average square distance of a point $x\in E$ to its representative $q(x)$ whenever $x$ is drawn from $P$.
The risk of $q$ characterizes therefore an important feature of the clustering scheme based on $q$ and defining optimality of $q$ in terms of the value of its risk appears as a reasonable approach.
However, an important though simple observation is that the excess risk $R(q)-R(q^{\star})$, for an optimal quantizer $q^{\star}$, isn't well suited to describe the geometric similarity between the clusterings based on $q$ and $q^{\star}$. For one thing, there might be several optimal codebooks. 
Also, even in the context where there is a unique optimal codebook, quite different configurations of centers $\mathbf c$ may give rise to very similar values of the excess risk $R(\mathbf c)-R(\mathbf c^{\star})$.
This observation relates to the difference between \emph{estimating the optimal quantizer} and \emph{learning to perform as well as the optimal quantizer} and is relevant in a more general context as briefly discussed in Appendix \ref{contrast} below.
Basically, the idea of stability we are referring to consists in identifying situations where having centers $\mathbf c$ with small excess risk guarantees that $\mathbf c$ isn't far from an optimal center $\mathbf c^{\star}$ geometrically speaking.
We formalize this idea below.

\begin{defi}Consider a function $F:\Qk\times \Qk\to\R_+$. The clustering problem discussed in subsections \ref{kmeans} and \ref{risk} is called $(F,\phi)$-stable if, for any optimal quantizer $q^{\star}$, for any auxiliary quantizer $q$,  
\begin{equation}\label{eq:stability}
F(q^{\star},q)\le \phi(R(q)-R(q^{\star})).
\end{equation}
We say that the clustering problem is \emph{strongly} stable for $F$, if $\phi$ is linear.
\end{defi}

Note first that, for some chosen $F$, the notion of stability defined above characterizes a property of the underlying distribution $P$.
Here, properties of the function $F$ are deliberately unspecified as, in practice, $F$ can be chosen in order to encode very different properties, of more or less geometric nature.
An important property of this notion is that stable clustering problem are such that $\e$-minimizers of the risk are "close" (in the sense of $F$) to the optimal quantizer (see Corollary \ref{cor:pn} below).

\begin{rem}
The notion of stability described above differs from the notion of algorithm stability studied in
\cite{ben2006sober} and \cite{ben2007stability}.
Their notion of stability is defined for a function (called algorithm) $A:\bigcup_nE^n\rightarrow \Qk$ that maps any data set $\{X_1,\dots,X_n\}$ to a quantizer $A(\{X_1,\dots,X_n\})$. In this context, the stability of $A$ is defined by
\[
{\rm Stab}(A,P)=\lim_{n\rightarrow\infty}\mathbf{E}D(A(\{X_1,\dots,X_n\}),A(\{Y_1,\dots,Y_n\})),
\]
where the $X_i$'s and $Y_i$'s are i.i.d. random variables of common distribution $P$ and $D$ is a (pseudo-) metric on $\Qk$. Then, an algorithm is said to be stable for $P$ if ${\rm \,Stab}(A,P)=0$.
According to this definition, any constant algorithm $A=q$ is stable.
A notable difference, is that our notion of stability includes a notion of consistency.
Indeed, since $q\mapsto R(q)$ is continuous (for a proper choice of the metric on $\Qk$), then our notion of stability measures (if and) at which rate $q\rightarrow q^\star$ whenever $R(q)\rightarrow R^\star$. Thus, we focus only on the behaviour of algorithms $A$ such that $R(A(\{X_1,\dots,X_n\}))\rightarrow R^\star$.
\end{rem}

A first rather obvious choice for $F$ is given by 
\begin{equation}
\label{dcenters}
F_{1}(q^{\star},q):= \min_{\sigma}\, \max_{1\le j\le k}d(c^{\star}_j,c_{\sigma(j)}),
\end{equation}
if $q(E)=\{c_1,\dots,c_k\}$ and $q^{\star}(E)=\{c^{\star}_1,\dots,c^{\star}_k\}$ and where the minimum is taken over all permutations $\sigma$ of $\{1,\dots,k\}$ (see Figure \ref{f1et2}). 

\begin{rem}
\label{rem:haussdorf}
Note that $F_{1}(q^{\star},q)$ does not always coincide with the Hausdorff distance $d_H(\mathbf c^{\star},\mathbf c)$ between $\mathbf c=\{c_1,\dots,c_k\}$ and $\mathbf c^{\star}=\{c^{\star}_1,\dots,c^{\star}_k\}$. Indeed, Figure \ref{bwdots} presents a configuration of codebooks $\mathbf c$ and $\mathbf c^{\star}$ that have small Hausdorff distance but define NN quantizers $q$ and $q^{\star}$ with large $F_{1}(q^{\star},q)$. However, it may be seen that inequality 
\[ d_H(\mathbf c^{\star},\mathbf c)\le F_{1}(q^{\star},q) \]
always holds and that, provided 
\[ d_{H}(\mathbf c^{\star},\mathbf c) < \frac{1}{2}\min_{i\ne j}|c^{\star}_i-c^{\star}_j|,\]
we obtain $d_H(\mathbf c^{\star},\mathbf c)= F_{1}(q^{\star},q)$. The proof of these statements is reported in Appendix \ref{A:hausdorff}.
\end{rem}

\begin{figure}[htbp]
\centering
\includegraphics[width=8cm]{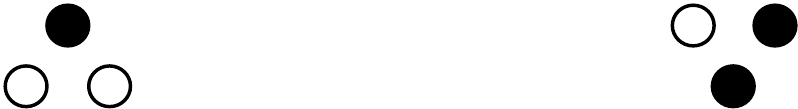}
\caption{In this simple case, where $k$=3, the set of black dots and the set of white dots have small Hausdorff distance but define two NN quantizers, say $q_1$ and $q_2$, for which $F_1(q_1,q_2)$ is large.}
\label{bwdots}
\end{figure}

Whenever $(E,|.|)$ is Euclidean, it follows from the previous remark and \citet{Po82b} that, provided the optimal codebook $\mathbf c^{\star}$ is unique, 
\[ F_1(q^{\star},\hat{q})\underset{n\to +\infty}{\longrightarrow} 0,\quad \mbox{a.s.},\]
when $\hat q$ is any quantizer minimizing the empirical risk $\hat R$. In \citet{Le15}, under the conditions of Theorem \ref{theoLevrard}, it is proven that for any optimal quantizer $q^{\star}$, and any $q\in\Qk$ such that $q(E)\subset\{x:|x|\le L\}$, 
\[
F_1(q^{\star},q)^2\le \frac{p_{\rm min}}{2}(R(q)-R(q^{\star})),
\]
provided $F_1(q^{\star},q)\le Br_0/4\sqrt{2}M$ which proves in this case (a local version of) the stability of the clustering scheme for $F_1$ (constants are defined in Theorem \ref{theoLevrard}).
In the same spirit, when $E=\mathbf R^d$ and for a measure $P$ with bounded support, \cite{rakhlin2007stability} show that $F_1(q_n,q'_n)\rightarrow 0$ as $n\rightarrow \infty$ whenever $q_n$ and $q'_n$ are optimal quantizers for  empirical measures $P_n$ and $P'_n$ whose supports differ by at most $o({\sqrt n})$ points.
In addition, their Lemma 5.1 shows that, for $P$ with bounded support,
\[
d_H(\bcs,\bc)\le C\,\esp [ | |X-q(X)|^2-|X-q^\star(X)|^2|]^{\frac{1}{d+2}},
\]
for some constant $C>0$. Note that,
since $\esp [ | |X-q(X)|^2-|X-q^\star(X)|^2|] \ge R(q)-R(q^\star)$, our main result (Theorem \ref{thm:absmarg}) improves this inequality under suitable conditions discussed below.

While $F_1$ captures distances between representatives of the two quantizers, it is however totally oblivious to the amount of wrongly classified points. From this point of view, a more interesting quantity is described by
\begin{equation}
\label{misclass}
F_2(q^{\star},q):=\min_{\sigma}\,P\left[\,\left(\,\bigcup_{j=1}^{k}V_j(\mathbf c^{\star})\cap V_{\sigma(j)}(\mathbf c)\,\right)^{c}\,\right],
\end{equation}
where the minimum is taken over all permutations $\sigma$ of $\{1,\dots,k\}$ (see Figure \ref{f1et2}).
This quantity measures exactly the amount of points that are misclassified by $q$ compared to $q^{\star}$, regarding $P$.\\

In the present paper, we study a related quantity, of geometric nature, defined simply as the average square distance between a quantizer $q$ and an optimal quantizer $q^{\star}$, \emph{i.e.}
\begin{equation}
\label{fatF}
\mathbf F(q^{\star},q)^2:=\int_E d(q(x),q^{\star}(x))^2\,\textrm{d}P(x).
\end{equation}
As discussed later in the paper (see Subsection \ref{sub:comp}), this quantity may be seen as an intermediate between $F_1$ and $F_2$ incorporating both the notion of proximity of the centers and the amount of misclassified points. The general concern of the paper will be to establish conditions under which the clustering scheme is strongly stable for this function $\mathbf F^2$.

\begin{figure}[htbp]
\centering
\includegraphics[width=6.4cm,angle=90]{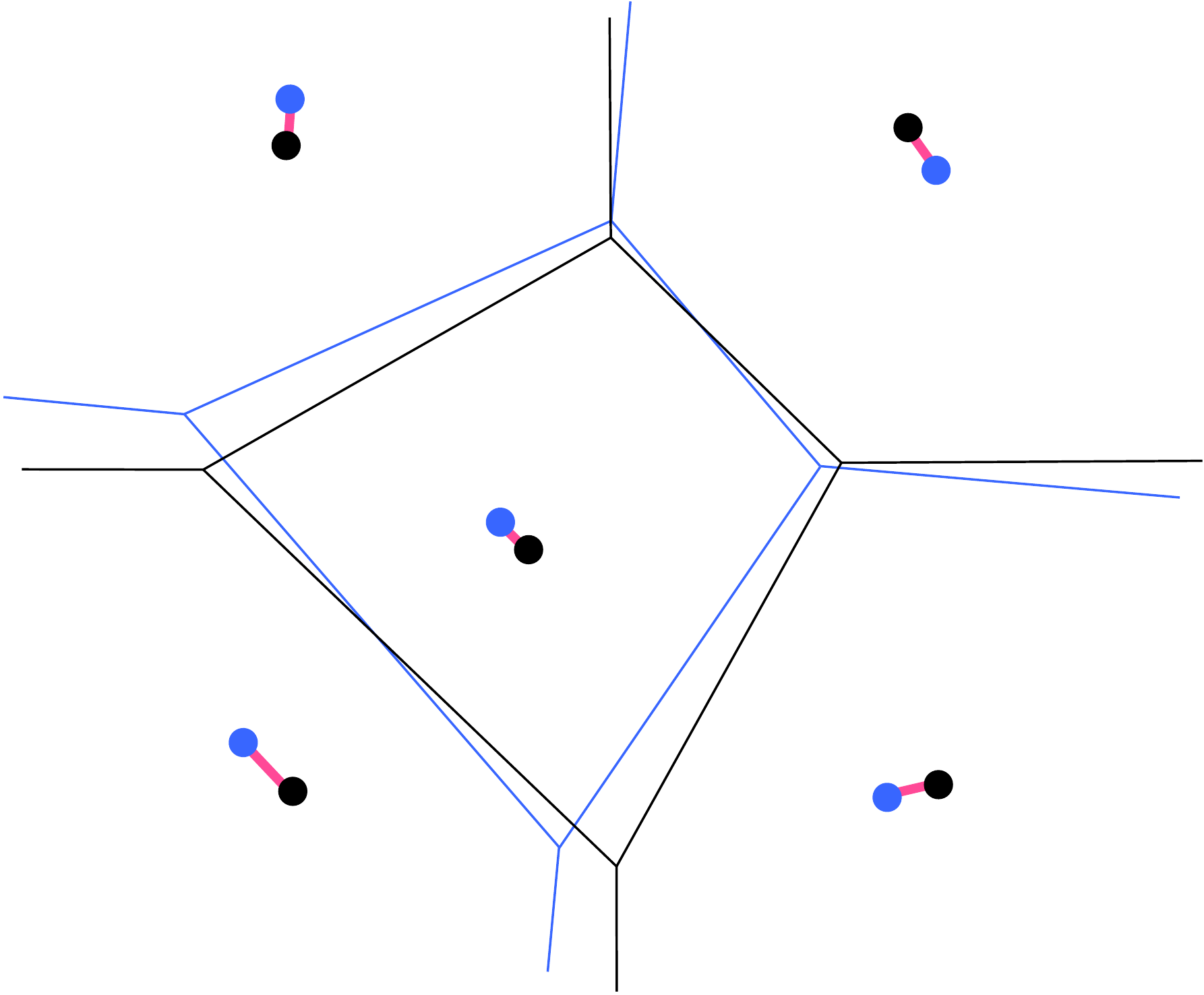}
\includegraphics[width=6.4cm,angle=90]{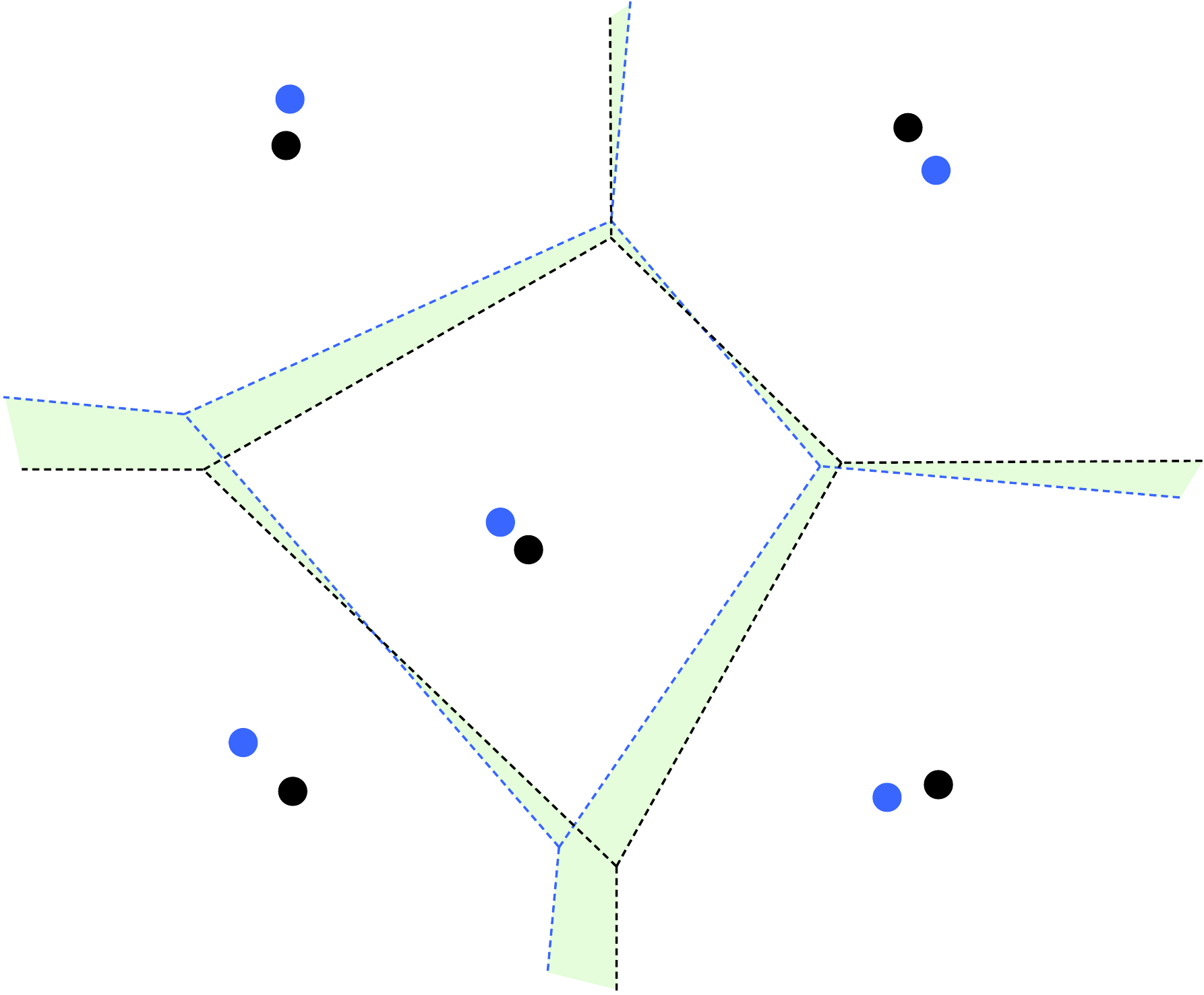}
\caption{The image of $q^{\star}$ (resp. $q$) is represented by the black (resp. blue dots). The quantity $F_1(q^{\star},q)$ corresponds to the length of the longest pink segment in the first (left) figure. The quantity $F_2(q^{\star},q)$ is the $P$ measure of the light green area in the second (right) figure.}
\label{f1et2}
\end{figure}

\section{Stability results} In this section, we present our main results. In the sequel, we restrict ourselves  to the case where $E$ is a (separable) Hilbert space  with scalar product $\langle.,.\rangle$ and associated norm $|.|$. For any $E$-valued random variable $Z$, we'll denote 
\[\|Z\|^2:=\esp|Z|^2,\]
for brevity.



\subsection{Absolute margin condition} We first address the issue of characterizing the stability of the clustering scheme in terms of the function $\mathbf F$ defined in \eqref{fatF}.
The next definition plays a central role in our main result. Recall that $X$ denotes a generic random variable with distribution $P$.

\begin{defi}[\,Absolute margin condition\,]\label{def:alambda}
Suppose that $\int|x|^2{\rm d}P(x)<+\infty$ and let $q^\star$ be an optimal $k$-points quantizer of $P$.
For $\lambda\ge 0$, define
\[
A(\lambda)=\left\{x\in E : q^\star(x+\lambda(x-q^\star(x)))=q^\star(x)\right\}.
\]
Then, $P$ is said to satisfy the \emph{absolute margin condition with parameter $\lambda_0>0$}, if both the following conditions hold:
\begin{enumerate}
\item $P(A(\lambda_0))=1$.
\item For any random variable $Y$ such that $\|X-Y\|\le \lambda_0\|X-q^\star(X)\|$, the map
\[
q\in\Qk\mapsto\|Y-q(Y)\|^2
\]
has a unique minimizer $q_{\lambda}$.
\end{enumerate}
\end{defi}

\begin{figure}[htbp]
\centering
\includegraphics[width=12cm]{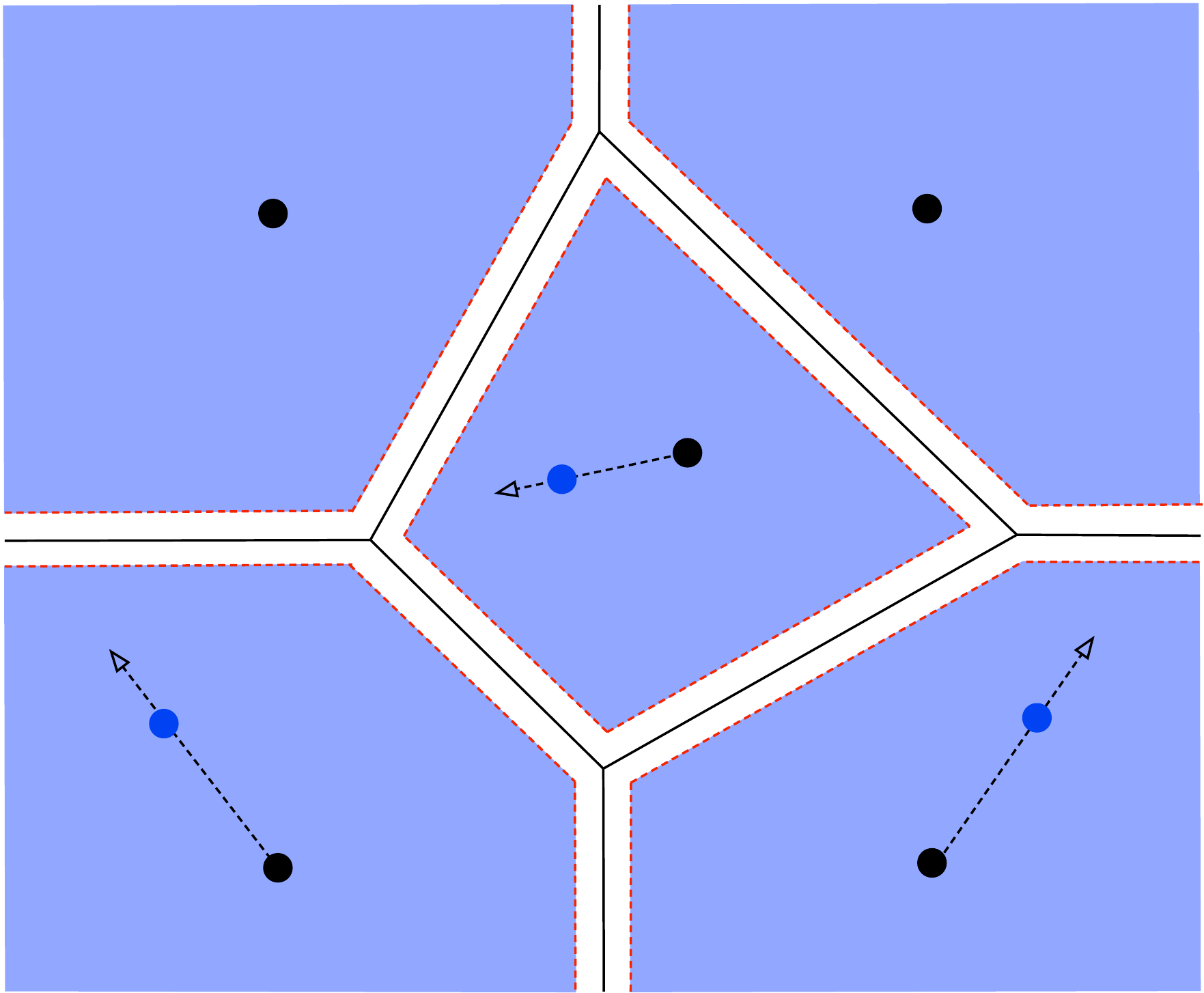}
\caption{The figure represents an optimal codebook (the black dots) for a distribution $P$, the frontier of the associated Voronoi diagram (the solid black line) and, for a fixed value of $\lambda>0$, the set $A(\lambda)$ (the light blue area). For points $x$ in $A(\lambda)$ (blue dots), the figure represents the associated point $x+\lambda(x-q^{\star}(x))$ (tip of the arrows) which, by definition of $A(\lambda)$, belongs to $A(\lambda)$.}
\label{Alambda}
\end{figure}

The second condition means that every probability measure, in a  neighborhood of $P$, has a unique $k$-quantizer.
Note that $A(0)=E$ and that $A(\lambda)\subset A(\lambda')$ for $\lambda'\le \lambda$. Letting $\mathbf c^{\star}=q^{\star}(E)$, the first point of this definition states that the neighborhood $E\setminus A(\lambda_0)$ of the frontier $\mathcal{F}(\bcs)$ is of probability zero (see Figure \ref{Alambda}). The next remark discusses the geometry of the set $A(\lambda)$, involved in the previous definition, in comparison with the sets $\mathcal{F}(\bcs)^t$ used in Definition \ref{def:marglev}. In particular, it follows from the following remark that, for appropriate $0<t_1<t_2$, the set $E\setminus A(\lambda)$ satisfies 
\[\mathcal F(\mathbf c)^{t_1}\subset (E\setminus A(\lambda_0)) \subset \mathcal F(\mathbf c)^{t_2}.\]

\begin{rem}
\label{lem:l0}
Let ${\mathbf c}=\{c_1,\dots,c_k\}\subset E$. Denote 
\[
m({\mathbf c})=\min_{i\ne j}| c_i-c_j|\quad\mbox{and}\quad M({\mathbf c})=\max_{i\ne j}| c_i-c_j|.
\]
For all $\lambda\ge 0$ and $t>0$, let 
\[
A(\lambda):=\{x\in E: q(x+\lambda(x-q(x)))=q(x)\}\et B(t):=E\setminus \mathcal F(\mathbf c)^{t}.\]
Then the following statements hold. 
\begin{enumerate}
\item For all $0< t<M({\mathbf c})/2$, 
\[
B(t)\subset A\left(\frac{2t}{M({\mathbf c})-2t}\right).
\]
\item For all $\lambda>0$
\[
A(\lambda)\subset B\left(\frac{m({\mathbf c})\lambda}{2(1+\lambda)}\right).
\]
\end{enumerate}
\end{rem}


We are now in position to state the main result of this paper.

\begin{theo}\label{thm:absmarg}
Suppose that $\int|x|^2{\rm d}P(x)<+\infty$. Let $q^{\star}$ be an optimal quantizer for $P$ and suppose that $P$ satisfies the absolute margin condition \ref{def:alambda} with parameter $\lambda_0>0$. Then, for any $q\in\Qk$, it holds that
\[
\mathbf F(q^{\star},q)^2\le \frac{1+\lambda_0}{\lambda_0}(R(q)-R(q^{\star})).
\]
\end{theo}

\begin{rem}
The above theorem states that the clustering scheme is strongly stable for $\mathbf F^2$ provided the absolute margin condition holds. Here, we briefly argue that this result is optimal in the sense that strong stability requires that both hypotheses of the absolute margin condition \ref{def:alambda} hold in general.
\begin{enumerate}
\item The following example shows that the first point of the absolute margin condition cannot be dropped.
Take $P$ uniform on $[-1,1]\times [-1/2,1/2]$ and fix $k=2$. Then the first point of the absolute margin condition is clearly not satisfied.
The codebook \[\mathbf c^\star=\{(-1/2,0),(1/2,0)\}\] defines the unique optimal quantizer.
For $\e >0$, consider now \[\mathbf c_\e=\{(-1/2,\e),(1/2,-\e)\}.\]
Then it can be checked through straightforward computations that $\mathbf F(q^\star,q_{\e})=\e$ and that $R(q_\e)-R(q^\star)\le \e^2$, so that there exists no $\lambda>0$ for which inequality \[\mathbf F(q^\star,q_{\e})^2\le \frac{1+\lambda}{\lambda}(R(q_\e)-R(q^\star))\] 
holds for all $\e>0$.\vspace{0.2cm}
\item If there is not uniqueness of an optimal quantizer of $P$, then the result clearly cannot hold. 
Although, this uniqueness property does not suffice.
To illustrate this statement, suppose $P$ is defined by $P=(\mu_1+\mu_2)/2$ where $\mu_1$ is uniform on $[-1;1]\times\{1\}$ and $\mu_2$ is uniform on $[-1;1]\times\{-1\}$.
For $k=2$, the codebook 
\[\mathbf c^{\star}=\{(0,1),(0,-1)\}\]
defines the unique optimal quantizer for $P$. The distribution $P$ satisfies the first point of the absolute margin condition for any $\lambda>0$, but fails to satisfies the second point for large $\lambda$.
In particular, it follows from details in the proof of Theorem \ref{thm:absmarg} that the desired inequality cannot hold for large $\lambda$.
\end{enumerate}
\end{rem}

An interesting consequence of Theorem \ref{thm:absmarg} holds in the context of empirical measures for which the absolute margin condition always holds. Consider a sample $X_1,\dots,X_n$ composed of i.i.d. variables with distribution $P$ and let 
\[P_n=\frac{1}{n}\sum_{i=1}^{n}\delta_{X_i}.\]
The next result ensures that an $\e$-empirical risk minimizer (i.e. a quantizer $q_\e$ such that $R_n(q_\e)\le \inf_q R_n(q)+\e$) is at a distance (in terms of $\mathbf F$) at most $\e(1+\lambda)/\lambda$ to an empirical risk minimizer for some $\lambda$ depending only on $P_n$.

\begin{cor}
\label{cor:pn}
Let $\e>0$.
Let $P_n$ be the empirical measure of a measure
$P$, associated with sample $\{X_1,\dots,X_n\}$.
Suppose $P_n$ has a  unique optimal quantizer $\hat q$.
Then $P_n$ satisfies the absolute margin condition for some $\lambda_n>0$. In addition, if $q_{\e}\in\Qk$ satisfies
\[
\frac{1}{n}\sum_{i=1}^n|X_i-q_{\e}(X_i)|^2\le \e+\frac{1}{n}\sum_{i=1}^n|X_i-\hat q(X_i)|^2,
\]
then
\[
\mathbf F(\hat q,q_{\e})^2\le \frac{1+\lambda_n}{\lambda_n} \e.
\]
\end{cor}

The last result follows easily from Theorem 4.2 in \citet{GrLu00} (stating that $P_n(\mathcal F(\hat\bc))=0$, for $\hat{\mathbf c}=\hat q(E)$, and thus $P_n(A(\lambda))=1$ for some $\lambda>0$) and from Theorem \ref{thm:absmarg}.
The proof is therefore omitted for brevity. The interpretation of this corollary is that any algorithm producing a quantizer $q$ with small empirical risk $\hat R(q)$ will be, automatically, such that $\mathbf F(\hat q, q)$ is small (and again, provided uniqueness of $\hat q$) if $\lambda_n$ is large. Parameter $\lambda_n$ defined by the absolute margin condition, thus provides a key feature for stability of the $k$-means algorithm.
A nice property of the previous result is that $\lambda_n$ is of course independent of the $\e$-minimizer $q_{\e}$.
However, an important remaining question, of large practical value, is to lower bound $\lambda_n$ with large probability to assess the size of the coefficient $(1+\lambda_n)/\lambda_n$. This is left for future research. 



\subsection{Comparing notions of stability}\label{sub:comp}

This subsection describes some relationships existing between the function $\mathbf F$ involved in our main result, with the two functions $F_1$ and $F_2$ mentioned earlier in section \ref{sec:introstab}. Below, we restrict attention to the case where there is a unique optimal quantizer $q^\star$. Comparing $\mathbf F$ and $F_2$ can be done straightforwardly. Let  \[m=\inf_{i\ne j}|c^\star_i-c^\star_j|\et M=\sup_{i\ne j}|c^\star_i-c^\star_j|.\] 
Observe that, for $F_1(q^{\star},q)$ small enough, the permutation reaching the minimum in the definitions of $F_1$ and $F_2$ is the same and can be assumed to be the identity without loss of generality. Then, it follows that, for $F_1(q^{\star},q)$ small enough, 
\begin{align*}
\mathbf F(q^\star,q)^2&=\sum_{i,j=1}^kP(V_i(\bcs)\cap V_j(\bc))|c^\star_i-c_j|^2\\
&\le \sum_{i=1}^kP(V_i(\bcs)\cap V_i(\bc))|c^\star_i-c_i|^2+ \sum_{i\ne j}P(V_i(\bcs)\cap V_j(\bc))(|c^\star_j-c_j|+M)^2\\
&\le  F_1(q^\star,q)^2 + F_2(q^\star,q)(F_1(q^\star,q)+M)^2,
\end{align*}
and similarly, when $m\ge F_1(q^\star,q)$,
\begin{align*}
\mathbf F (q^\star,q)^2&\ge \sum_{i=1}^kP(V_i(\bcs)\cap V_i(\bc))|c^\star_i-c_i|^2+ \sum_{i\ne j=1}^kP(V_i(\bcs)\cap V_j(\bc))(m-|c^\star_j-c_j|)^2\\
&\ge  F_2(q^\star,q)(m-F_1(q^\star,q))^2.
\end{align*}
This two inequalities imply that $\mathbf F^2$ and $F_2$ are comparable whenever $F_1$ is small enough.

Comparing $F_1$ and $\mathbf F$ requires more effort, although one inequality is also quite straightforward.
Recall the notation $p_{\min}=\inf_iP(V_i(\bcs))$.
Suppose again that the optimal permutation in the definition of $F_1$ is the identity.
Then, remark that $F_1(q^\star,q)\le m/2$, implies $|c_i^\star-c_i|\le |c_i^\star-c_j|$, for all $i,j$. Thus, in this case,
\begin{align*}
\mathbf F(q^\star,q)^2&=\mathbf E |q^\star(X)-q(X)|^2\\
&=\sum_{i,j=1}^k P(V_i(\bcs)\cap V_j(\bc))|c^\star_i-c_j|^2\\
&\ge\sum_{i=1}^k \sum_{j=1}^k P(V_i(\bcs)\cap V_j(\bc))|c^\star_i-c_i|^2\\
&\ge p_{\min}F_1(q^\star,q)^2.
\end{align*}

In view of providing a more detailled result, we define the function $p^\star$, similar in nature to the function $p$ introduced by \cite{Le15} and defined in \ref{def:marglev}.
\begin{defi}\label{def:pstar}
For a metric space $(E,d)$ and a probability measure $P$ on $E$, let $X$ be a random variable of law $P$.
Denote $q^\star$ an optimal quantizer of $P$ with image $\bcs=\{c^\star_1,\dots,c^\star_k\}$ and $\partial V_i(\bcs)$ the frontier of the Voronoi cell associated to $c_i$.
Then, for all $t>0$, we let
\[
p^\star(t):=\mathbf{P}\left(\bigcup_{i=1}^k\left\{ md(X,\partial V_i(\bcs))\le 2d(X,q^\star(X)) t +2t^2\right\}\right),
\]
where $m=\inf_{i\ne j}|c^\star_i-c^\star_j|$.
\end{defi}
While $p(t)$ corresponds to the probability of the $t$-inflated frontier of the Voronoi cells (defined in Definition \ref{def:marglev}), $p^\star(t)$ corresponds to a similar object in which the inflation of the frontier gets larger as the points go further from their representant in the codebook $\bf c^{\star}$. These two functions can thus differ significantly, in general.
However, since $m/4\le d(X,q^\star(X))$ for $X$ such that $d(X,\partial V_i(\bcs))<m/4$, it follows that
\[
p(t)\le p^\star\left(2t\right),
\]
whenever $0<t<m/4$.
And when the probability measure $P$ has its support in a ball of diameter $R>0$, it can be readily seen that for all $t>0$
\[
p^\star(t)\le p\left(m^{-1}\left[2R t + 2t^2\right]\right).
\]
If the support of $P$ is not contained in a ball, the comparison is not as straightforward.

We can now state the last comparison inequality.
\begin{pro}\label{thm:pstar}
Under the same setting as in the Definition \ref{def:pstar},
\[
\mathbf F(q^\star,q)^2\le F_1(q^\star,q)^2+p^\star(F_1(q^\star,q))(M+F_1(q^\star,q))^2
\]
\end{pro}
A consequence of this proposition and the result of \cite{Le15} recalled in Theorem \ref{theoLevrard} is the following
\begin{cor}
Under the conditions of Theorem \ref{theoLevrard}, 
\[
\mathbf F(q^\star,\hat q)^2=\mathcal{O}\left(\frac{1}{n}\right)+p^\star\left(\mathcal{O}\left(\frac{1}{\sqrt{n}}\right)\right),
\]
for any empirical risk minimizer $\hat q$.
\end{cor}

\section{Proofs} This section gathers the proofs of the main results of the paper. Additional proofs are postponed to the appendices.

\subsection{Proof of Theorem \ref{thm:absmarg}}
Recall that $E$ is a Hilbert space with scalar product $\langle.,.\rangle$, norm $|.|$ and that, for an $E$-valued random variable $Z$ with square integrable norm, we denote $\|Z\|^2=\esp|Z|^2$ for brevity. For $\lambda>0$, set 
\[x_{\lambda}=x+\lambda(x-q^\star(x)).\] 
As $E$ is a Hilbert space, we have for all $y,z\in E$ and all $t\in [0,1]$, 
\[|ty+(1-t)z|^2=t|y|^2+(1-t)|z|^2+t(1-t)|y-z|^2.\]
Now for all $x\in E$, any quantizer $q\in\Qk$ and any $\lambda >0$, using the previous inequality with $y=x_{\lambda}-q(x)$, $z=q^{\star}(x)-q(x)$ and $t=(1+\lambda)^{-1}$, it follows that 
\begin{align*}
|q^\star(x)-q(x)|^2&=\frac{1+\lambda}{\lambda}(|x-q(x)|^2-|x-q^\star(x)|^2)+\frac{|x_{\lambda}-q^{\star}(x)|^2-|x_{\lambda}-q(x)|^2}{\lambda}\\
&\le \frac{1+\lambda}{\lambda}(|x-q(x)|^2-|x-q^\star(x)|^2)+\frac{|x_{\lambda}-q^{\star}(x)|^2-|x_{\lambda}-q(x_{\lambda})|^2}{\lambda},
\end{align*}
where the last inequality follows from the fact that $q$ is a nearest neighbor quantizer. Integrating this inequality with respect to $P$, we obtain
\begin{equation}
\label{eq:equalitygeo}
\mathbf F(q^{\star},q)^2\le \frac{1+\lambda}{\lambda}(R(q)-R(q^{\star}))+\frac{1}{\lambda}c_{q}(\lambda),
\end{equation}
where we have denoted 
\[
c_{q}(\lambda):=\|X_{\lambda}-q^{\star}(X)\|^2-\|X_{\lambda}-q(X_\lambda)\|^2.
\]
Observe that $\lambda\mapsto c_q(\lambda)$ is continuous.
Now, define 
\[
c_\infty(\lambda):=\sup_q c_q(\lambda),
\]
where the supremum is taken over all $k$-points quantizers $q\in\Qk$.
The function $\lambda\mapsto c_\infty(\lambda)$ satisfies obviously $c_\infty(\lambda)\ge c_{q^\star}(\lambda)\ge 0$, for all $\lambda>0$. To prove the theorem, we will show that $c_\infty(\lambda_0)\le0$, whenever $P$ satisfies the absolute margin condition with paramater $\lambda_0>0$. To that aim, we provide two auxiliary results.

\begin{lem}\label{lem:main}
Suppose there exists $R>0$ such that $P(B(0,R))=1$.
For all $\lambda>0$, denote $q_\lambda$ any quantizer such that $c_{q_\lambda}(\lambda)=c_\infty(\lambda)$ and denote $q^\star$ an optimal quantizer of the law of $X$.
Suppose the absolute margin condition holds for $\lambda_0>0$.
Then, for all $0<\lambda_1<\lambda_0$, there exists $\e>0$ such that for all $0<\lambda<\lambda_1$, if $F_1(q_\lambda,q^\star)< \e$, then
\[
q^\star=q_\lambda.
\]
\end{lem}

\begin{proof}[Proof of lemma \ref{lem:main}]
The main idea of the proof is that since the Voronoi cells are well separated (inflated borders are with probability $0$), when a quantizer is close enough to the optimal one, it shares its Voronoi cell (on the support of $P$) and thus, centroid condition requires that quantizer have to be centroid of its cell to be optimal.

Set $q_{\lambda}(E)=\{c_1,...,c_k\}$ and $\{c_1^\star,\dots,c_k^\star\}=q^\star(E)$.
Suppose without loss of generality that the optimal permutation in the definition of $F_1$ is the identity.
The assumption implies that, with probability one, for each $1\le i \le k$, on the event $q^\star(X)=c_i^\star$, the inequality $|X_{\lambda_0}-c^\star_i|^2\le |X_{\lambda_0}-c^\star_j|^2$ holds, or equivalently
\begin{equation}\label{eq:lambdaineg}
2(1+{\lambda_0})\langle X-c_i^\star,c_j^\star-c_i^\star\rangle \leq|c_i^\star-c_j^\star|^2.
\end{equation}

However,
\begin{align*}
|X_{\lambda_0}-c_i|^2 =& (1+{\lambda_0})^2|X-c_i^\star|^2+|c_i^\star-c_i|^2+2(1+{\lambda_0})\langle X-c_i^\star,c_i^\star-c_i\rangle \\ 
|X_{\lambda_0}-c_j|^2 =& (1+{\lambda_0})^2|X-c_i^\star|^2+|c_i^\star-c_j|^2+2(1+{\lambda_0})\langle X-c_i^\star,c_i^\star-c_j\rangle  
\end{align*}
so that $|X_{\lambda_0}-c_i|^2\le |X_{\lambda_0}-c_j|^2$ if
\[
2(1+{\lambda_0})\langle X-c^\star_i,c_j-c_i\rangle \leq |c_i^\star-c_j|^2-|c_i^\star-c_i|^2.
\]
Since \eqref{eq:lambdaineg} holds, for all $\lambda_1 < \lambda_0$, there exists therefore $\e=\e({\lambda_0},\lambda_1,R,\max\{|c_i^\star-c_j^\star|:i\ne j\})$ such that, if $F_1(q^\star,q)< \e$, then for all $\lambda\le\lambda_1$,
\[
|X_{\lambda}-c_i|^2< |X_{\lambda}-c_j|^2,
\]
on the event $q^\star(X)=c_i^\star$.
As a result,
\[
\mathbf P\left(\bigcup_{i=1}^k\{q^{\star}(X)=c^\star_i\}\cap \{ q_{\lambda}(X_\lambda)=c_i\}\right)=1.
\]
This means that $q^\star$ and $q_{\lambda}$ share the same cells on the support of $P$.
Thus, 
\begin{align}
\|X_\lambda-q_\lambda(X)\|^2& =(1+\lambda)^2\sum_{i=1}^k\mathbf E\mathbf 1_{\left\{q^\star(X)=c^\star_i\right\}} \left|X-\frac{\lambda c^\star_i + c_i}{1+\lambda}\right|^2\nonumber\\
&\ge (1+\lambda)^2 \sum_{i=1}^k \mathbf{E} \mathbf 1_{\left\{q^\star(X)=c^\star_i\right\}} \left|X- c^\star_i\right|^2\label{eq:ineqcentr}\\
& = \|X_\lambda-q^{\star}(X)\|^2,
\nonumber
\end{align}
where inequality \eqref{eq:ineqcentr} follows from the center condition \eqref{ccond}.
Therefore, since $q_\lambda$ minimizes $\|X_\lambda-q(X_\lambda)\|^2$ amongst NN quantizers, \eqref{eq:ineqcentr} is an equality, so that $c_i=c^\star_i$ i.e. $q^\star=q_\lambda$; since $\|X_\lambda-X\|=\lambda\|X-q^\star(X)\|\le \lambda_0\|X-q^\star(X)\|$ implies from absolute margin condition that $q_\lambda$ is unique.
\end{proof}

\begin{lem}\label{lem:nonuniq}
Suppose $X$ satisfies the conditions of Lemma \ref{lem:main}.
Denote 
\[
\lambda^-=\min\{\lambda : c_\infty(\lambda)>0\}.
\]
Then $\lambda^-\ge \lambda_0$.
\end{lem}

\begin{proof}[Proof of lemma \ref{lem:nonuniq}]
The idea of the proof of this lemma is that uniqueness condition of the margin condition implies continuity of the optimal quantizer with respect to $\lambda$ and previous lemma states that the only optimal quantizers for $X_\lambda$ that is close to $q^\star$ is $q^\star$.

Suppose $\lambda^-<\lambda_0$. Then, by second point of the margin condition (Definition \ref{def:alambda}), there exists only one quantizer $q_{\lambda^-}$ such that $c_{\infty}(\lambda^{-})=c_{q_{\lambda^-}}(\lambda^-)$.
By definition of $\lambda^-$, $c_{q_{\lambda^-}}(\lambda)\le 0$ for $\lambda<\lambda^-$. Therefore, by continuity of $\lambda\mapsto c_{q_{\lambda^-}}(\lambda)$,  
\[
c_{q_{\lambda^-}}(\lambda^-)=0,
\]
and thus; by previous Lemma \ref{lem:main}, $q_{\lambda^-}=q^\star$, since $c_{q^\star}(\lambda^-)=0$.
Now, for all $\lambda>\lambda^-$, denote by $q_\lambda$ any quantizer such that $c_{q_\lambda}(\lambda)=c_\infty(\lambda)$, which exists by Lemma \ref{lem:cont}.
Then by Lemma \ref{lem:cont} again, $F_1(q_\lambda, q_{\lambda^-})\rightarrow 0$ as $\lambda\rightarrow\lambda^-$, so that for all $\lambda-\lambda^->0$ small enough, Lemma \ref{lem:main} applies and states $q_\lambda=q_{\lambda^-}=q^\star$; which contradicts the definition of $\lambda^-$.
\end{proof}

Therefore, $c_\infty(\lambda_0)=0$, and thus \eqref{eq:equalitygeo} gives
\[
\mathbf F(q^{\star},q)^2\le\frac{1+\lambda_0}{\lambda_0}(R(q)-R(q^{\star})).
\]

Finally, by a continuity argument, the result still holds without the assumption of boundedness $\mathbf{P}(|X|\le R)=1$.

\subsection{Proof of Proposition \ref{thm:pstar}}
The following proof borrows some arguments from the proof of Lemma 4.2 of \cite{Le15}. Recall that $m=\inf_{i\ne j} |c^\star_i-c^\star_j|$.
Take $1\le i,j\le k$ and consider the hyperplane \[h^\star_{i,j}:=\{x\in E: |x-c^\star_i|=|x-c^\star_j|\}.\] 
Then, for all $x\in V_i(\mathbf c^{\star})$,
\begin{align}
d(x,h^\star_{i,j})&=\frac{|\langle c^\star_i+c^\star_j-2x,c^\star_i-c^\star_j\rangle|}{2|c^\star_i-c^\star_j|}\nonumber\\
&\le\frac{|\langle c^\star_i+c^\star_j-2x,c^\star_i-c^\star_j\rangle|}{2m}\nonumber\\
&=\frac{|x-c^\star_j|^2-|x-c^\star_i|^2}{2m}.
\label{thm:pstar:e1}
\end{align}
Without loss of generality, suppose now for simplicity that the permutation $\sigma$ achieving the minimum in the definition of $F_1(q^{\star},q)$ is the identity, $\sigma(j)=j$, so that 
\[F_1(q^{\star},q)=\max_i|c^{\star}_i-c_i|.\]
Then, it follows that for $x\in V_i({\bf c^\star})\cap V_j({\bf c})$,
\begin{align}
|x-c^\star_j|^2-|x-c^\star_i|^2 &\le (|x-c_j|+|c_j-c^\star_j|)^2-|x-c^\star_i|^2\nonumber\\
&\le(|x-c_i|+|c_j-c^\star_j|)^2-|x-c^\star_i|^2\nonumber\\
&\le(|x-c^\star_i|+|c_i-c^\star_i|+|c_j-c^\star_j|)^2-|x-c^\star_i|^2\nonumber\\
&=2|x-c^\star_i|(|c_i-c^\star_i|+|c_j-c^\star_j|)+(|c_i-c^\star_i|+|c_j-c^\star_j|)^2\nonumber\\
&\le4|x-c^\star_i|F_1(q^{\star},q)+4F_1(q^{\star},q)^2.\label{eq:forpstar}
\end{align}
Thus, using the fact that, for all $x\in V_i(\mathbf c^{\star})$, we have
\[d(x,\partial V_i(\mathbf c^{\star}))=\min_{i\ne j}d(x,h^{\star}_{i,j}),\]
we deduce from the previous observations that, for all $i\ne j$,
\[
V_i({\bf c^\star})\cap V_j({\bf c})\subset\left\{ x\in E: md(x,\partial V_i(\mathbf c^\star))\le 2|x-q^\star(x)| F_1(q^{\star},q) + 2F_1(q^{\star},q)^2\right\}.
\]
The right hand side being independent of $j$, we obtain in particular,
\[
\bigcup_{j\ne i}V_i({\bf c^\star})\cap V_j({\bf c})\subset\left\{ x\in E: md(x,\partial V_i(\mathbf c^\star))\le 2|x-q^\star(x)| F_1(q^{\star},q) + 2F_1(q^{\star},q)^2\right\}.
\]
Therefore,
\begin{align*}
\mathbf{E}|q^\star(X)-q(X)|^2&=\sum_{i,j=1}^k P(V_i(\bcs)\cap V_j(\bc))|c^\star_i-c_j|^2\\
&=\sum_{i=1}^k P(V_i(\bcs)\cap V_i(\bc))|c^\star_i-c_i|^2+\sum_{i\ne j, i=1, j=1}^k P(V_i(\bcs)\cap V_j(\bc))|c^\star_i-c_j|^2\\
&=\sum_{i=1}^k P(V_i(\bcs)\cap V_i(\bc))|c^\star_i-c_i|^2+\sum_{i\ne j, i=1, j=1}^k P(V_i(\bcs)\cap V_j(\bc))(|c^\star_j-c_j|+M)^2\\
&\le F_1(q^\star,q)^2+p^\star(F_1(q^\star,q))(F_1(q^\star,q)+M)^2
\end{align*}
which shows the desired result.

\appendix
\section{Technical results} 

\subsection{Proofs for Remark \ref{rem:haussdorf}}
\label{A:hausdorff}
Recall that, for any two sets $A$, $B\subset E$, their Hausdorff distance is defined by
\[d_{H}(A,B):=\inf\{\,\e>0:\, A\subset B^{\e}\mbox{ and } B\subset A^{\e}\,\},\]
where $A^{\e}=\{x\in E:d(x,A)\le \e\}$. The fact that $d_H(\mathbf c^{\star},\mathbf c)\le F_{1}(q^{\star},q)$ then follows easily from definitions. Now, to prove the second statement, observe that, in the context of the finite sets $\mathbf c$ and $\mathbf c^{\star}$, the infimum in the definition of $\delta:=d_H(\mathbf c^{\star},\mathbf c)$ is attained so that, for any $i\in\{1,\dots,k\}$, there exists some $j\in\{1,\dots,k\}$ such that $ c^{\star}_j\in B(c_i,\delta)=\{x\in E:|x-c_i|\le \delta\}$. Now suppose that 
\[ \delta <\frac{1}{2}\min_{i\ne j}|c^{\star}_i-c^{\star}_j|.\]
Then, the balls $B(c_i,\delta)$ are necessarily disjoint and therefore contain one and only one element of $\mathbf c^{\star}$, denoted $c^{\star}_{\sigma(i)}$. As a result, \[ F_1(q^{\star},q)\le \max_{i}|c_i-c^{\star}_{\sigma(i)}|=\delta,\]
where the last equality follows by construction. This implies the desired result.

\subsection{Proof for Remark \ref{lem:l0}}
Let $x\in E$ and denote $c_i=q(x)$. First, it may be checked that assumption $d(x,\mathcal F(\bc))> \e$ holds if, and only if, 
\begin{equation}
\label{lem:l0:e1}
\forall\,j\ne i:\quad\frac{\langle x-c_i,c_j-c_i\rangle}{|c_j-c_i|} < \frac{|c_j-c_i|}{2}-\e.
\end{equation}
Similarly, observe that $q(x_\lambda)=q(x)$ if and only if, for $j\ne i$, we have $| x_{\lambda}-c_i|< | x_{\lambda}-c_j|$. Using the definition of $x_{\lambda}$, this last condition may be equivalently written, for all $j\ne i$, as 
\begin{align}
(1+\lambda)^2|x- c_i|^2 &< |x-c_j|^2+2\lambda\langle x-c_i,x-c_j\rangle +\lambda^2|x- c_i|^2
\nonumber\\
& = (1+\lambda^2)|x-c_i|^2 +2\langle x-c_i,\lambda(x-c_j)+c_i-c_j\rangle+|c_i-c_j|^2.
\label{lem:l0:e2}
\end{align}
After simplification in \eqref{lem:l0:e2}, we therefore obtain that $q(x_\lambda)=q(x)$ if, and only if, 
\begin{align}
\forall\,j\ne i:\quad 0&<  |c_i-c_j|-2(1+\lambda)\frac{\langle x-c_i,c_j-c_i\rangle}{|c_j-c_i|}.
\label{lem:l0:e3}
\end{align}
The result now easily follows from combining \eqref{lem:l0:e1} and \eqref{lem:l0:e3}.

\subsection{A consistency result} The next result is adaptated from Theorems 4.12 and 4.21 in \citealp{GrLu00}.
\begin{lem}
\label{lem:cont}
Suppose $X\in L^2(\prob)$. Then, letting $q^{\star}$ be an optimal $k$-points quantizer for the distribution of $X$ and denoting $X_{\lambda}=X+\lambda (X-q^{\star}(X))$, the following statements hold.
\begin{enumerate}
\item For any $\lambda\ge 0$, there exists a $k$-points NN quantizer $q_{\lambda}$ such that 
\[\|X_{\lambda}-q_{\lambda}(X_\lambda)\|^2=\min_{q}\|X_{\lambda}-q(X_\lambda)\|^2,\]
where the minimum is taken over all $k$-points quantizers.
\item For all $\lambda_0\ge 0$, if $q_{\lambda_0}$ is unique,
\[
\lim_{\lambda\to\lambda_0}F_1(q_{\lambda},q_{\lambda_0})=0.
\]
\end{enumerate}
\end{lem}

\begin{proof}[Proof of lemma \ref{lem:cont}]
We state the result for a measure with bounded support and refer to \cite{GrLu00} for unbounded case.
\begin{enumerate}
    \item Let $q_n$ be a sequence of quantizers such that
    \[
    \|X_\lambda-q_n(X)\|^2\rightarrow \inf_q\|X_\lambda - q(X)\|^2,
    \]
    
    as $n\rightarrow \infty$. Since balls in $E$ are weakly compact, the centers $q_n(E)=\{c_1^n,\dots,c_k^n\}$ weakly converge to some limit $\{c_1,\dots,c_k\}$ up to a subsequence.
    Denote $q_0(X_\lambda)$ a limit of a weakly converging subsequence of $q_n(X_\lambda)$, realizing the limit $\liminf\|X-q_n(X_\lambda)\|^2$ then, by Fatou Lemma,
    \begin{align*}
        \liminf \|X_\lambda-q_n(X_\lambda)\|^2\ &\ge \mathbf E \liminf|X_\lambda-q_n(X_\lambda)|^2\\
        &= \|X_\lambda - q_0(X_\lambda)\|^2 + \mathbf E \liminf |q_0(X_\lambda)-q_n(X_\lambda)|^2\\
        &\ge \inf_q\|X_\lambda - q(X_\lambda)\|^2+ \mathbf E \liminf |q_0(X_\lambda)-q_n(X_\lambda)|^2,
    \end{align*}
    which shows that $q_0$ realizes the minimum of $\inf_q\|X_\lambda - q(X)\|^2$.\\
\item Similarly, for any sequence $\lambda_n\rightarrow\lambda_0$ as $n\rightarrow\infty$, $q_{\lambda_n}(X_{\lambda_n})$ has a weak limit $q_0(X_{\lambda_0})$ (up to subsequence).
    Then,
    \begin{align*}
        \|X_{\lambda_0}-q_{\lambda_0}(X_{\lambda_0})\|^2&= \liminf_{n\rightarrow\infty} \|X_{\lambda_n}-q_{\lambda_0}(X_{\lambda_n})\|^2\ \\
        &\ge \liminf_{n\rightarrow\infty} \|X_{\lambda_n}-q_{\lambda_n}(X_{\lambda_n})\|^2\ \\
        &\ge \mathbf E \liminf|X_{\lambda_n}-q_{\lambda_n}(X_{\lambda_n})|^2\\
        &\ge \|X_{\lambda_0} - q_0(X_{\lambda_0})\|^2 + \mathbf E \liminf |q_0(X_{\lambda_0})-q_{\lambda_n}(X_{\lambda_n})|^2\\
        &\ge \|X_{\lambda_0} - q_{\lambda_0}(X_{\lambda_0})\|^2 + \mathbf E \liminf |q_0(X_{\lambda_0})-q_{\lambda_n}(X_{\lambda_n})|^2.
    \end{align*}
    The last inequality holds because $q_{\lambda_0}$ is optimal.
    This shows 
    $\mathbf E \liminf|q_0(X)-q_n(X)|^2= 0$, and since $q_{\lambda_0}$ is supposed to be unique $q_0=q_{\lambda_0}$ and therefore, every subsequence converges to the same limit $q_{\lambda_0}$; so $F_1(q_\lambda,q_{\lambda_0})\rightarrow 0$.

\end{enumerate}
\end{proof}

\section{Stability of a learning problem}
\label{contrast}
In this section, we briefly argue that the problem considered in the paper, while of special interest in the context of unsupervised learning, finds a natural extension in a more general framework of learning theory, namely the context of contrast minimization. Let $\mathcal Z$ be a measurable space equipped with a probability distribution $P$ and let $T$ be a given set of parameters. Suppose given a sample $Z_1,\dots,Z_n$ of i.i.d. variables with common distribution $P$. Given a contrast function 
\[\mathcal C:\mathcal Z\times T\to \R_+,\] 
consider the problem of designing a data driven $t$, based on the sample $Z_1,\dots,Z_n$, achieving a small value of the risk function 
\[
R(t):= \int\mathcal C(z,t)\,\textrm{d}P(z).
\]
This general problem, known as contrast minimization, is a classical way to unify the supervised and unsupervised learning approaches as illustrated in the next example.

\begin{exm}
Classical examples include the following.
\begin{itemize}
\item \emph{\textbf{Supervised learning.}}  The supervised learning problem corresponds to the contrast minimization problem  where $\mathcal Z=\mathcal X\times \mathcal Y$, where $T$ is a class of candidate functions $t:\mathcal X\to \mathcal Y$ and where, for a given loss function $\ell:\mathcal Y^2\to\R_+$, the contrast is 
\[\mathcal C((x,y),t)=\ell(y,t(x)).\]
\item \emph{\textbf{Unsupervised learning.}}  The unsupervised learning problem discussed earlier in the present paper corresponds to the contrast minimization problem where $\mathcal Z$ is a metric space $(E,d)$, where $T$ is the set $\mathfrak Q$ of all $k$-points quantizers, for a given integer $k$, and where the contrast function is 
\begin{equation}
\label{contrast2}
\mathcal C(x,q)=d(x,q(x))^2.
\end{equation}
\end{itemize}
\end{exm}

Given the general problem of contrast minimization, formulated above, one may naturally extend the question discussed in the present paper by considering the following notion of stability.

\begin{defi}
Consider a function $F:T^2\to \R_+$ and an increasing function $\phi:\R_+\to\R_+$. Then, the contrast minimization problem is called $(F,\phi,\e)$-stable if, for any $t^{\star}$ minimizing the risk on $T$, 
\[F(t^{\star},t)\le\e\quad\Rightarrow\quad F(t^{\star},t)\le \phi(R(t)-R(t^{\star})).\]
\end{defi}

Our main result, Theorem \ref{thm:absmarg}, proves the stability of the contrast minimization problem for the contrast function defined in \eqref{contrast2}. The following result proves the stability of the supervised learning problem for a strongly convex loss function.

\begin{exm}
Consider the supervised learning problem described in the example above. Suppose there exists $\alpha>0$ such that, for all $y\in\mathcal Y$, the function $u\in\mathcal Y\mapsto\ell(y,u)$ is $\alpha$-strongly convex. Then, for any convex class $T$ of functions $t:\mathcal X\to \mathcal Y$ and any $t^{\star}$ minimizing the risk on $T$, we have
\[\int(t-t^{\star})^2\,\textrm{d}\mu\le \frac{4}{\alpha}(R(t)-R(t^{\star})),\]
for any $t\in T$, where $\mu$ is the marginal of $P$ on $\mathcal X$. In particular, for all $\e>0$, this learning problem is $(\e,\phi)$-stable for the $L^2(\mu)$ metric with $\phi(u)=2\sqrt{u}/\sqrt{\alpha}.$
\end{exm}

\bibliography{bibKMEANS}
\end{document}